\documentclass[12pt]{amsart}
\usepackage{ifthen,verbatim}
\usepackage{floatflt,graphicx,graphics}
\usepackage{subfig}

 \usepackage{tikz} 
\usepackage{mathrsfs}
\usepackage{amsmath,amssymb,amsthm}
\usepackage{mathtools}

\numberwithin{equation}{section}
\nonstopmode
\numberwithin{equation}{section}
\theoremstyle{plain}
\newtheorem{theorem}[equation]{Theorem}
\newtheorem{conjecture}[equation]{Conjecture}
\newtheorem{lemma}[equation]{Lemma}
\newtheorem{corollary}[equation]{Corollary}

\newtheorem{proposition}[equation]{Proposition}

\theoremstyle{definition}
\newtheorem{definition}[equation]{Definition}
\newtheorem{example}[equation]{Example}
\newtheorem{remark}[equation]{Remark}

\theoremstyle{remark}

\newcommand{\R}{\mathbb{R}}
\newcommand{\Z}{\mathbb{Z}}
\newcommand{\C}{\mathbb{C}}

\newcommand{\B}{\mathbb{B}}
\newcommand{\M}{\mathsf{M}}

\newcommand{\capa}{\mathrm{cap}\,}
\newcommand{\tnh}{\mathrm{th}}

{\qed\bigskip}

\newcounter{alphabet}

\newcounter{minutes}
\divide\time by 60
\newcounter{hours}
\multiply\time by 60
\addtocounter{minutes}{-\time}

\begin{document}
\bibliographystyle{amsplain}
\title
{
Intrinsic metrics in ring domains
}

\def\thefootnote{}
\footnotetext{
\texttt{\tiny File:~\jobname .tex,
          printed: \number\year-\number\month-\number\day,
          \thehours.\ifnum\theminutes<10{0}\fi\theminutes}
}
\makeatletter\def\thefootnote{\@arabic\c@footnote}\makeatother

\author[O. Rainio]{Oona Rainio}
\address{Department of Mathematics and Statistics, University of Turku, FI-20014 Turku, Finland}
\email{ormrai@utu.fi}

\begin{abstract}
Three hyperbolic type metrics including the triangular ratio metric, the $j^*$-metric and the M\"obius metric are studied in an annular ring. The Euclidean midpoint rotation is introduced as a method to create upper and lower bounds for these metrics, and their sharp inequalities are found. A new M\"obius-invariant lower bound is proved for the conformal capacity of a general ring domain by using a symmetric quantity defined with the M\"obius metric. 
\end{abstract}
\maketitle

\section{Introduction}

Given two points $x,y$ in a domain $G$, their \emph{intrinsic distance} indicates how these points are located with respect to both each other and the domain's boundary $\partial G$. One way to measure these kinds of distances is the hyperbolic metric, but there are also numerous other hyperbolic type and intrinsic metrics that can be used. In this article, we focus on three different metrics to study the intrinsic geometry of the annular ring $R(r,1)=\{z\in\C\text{ }|\text{ }r<|z|<1\}$ with $0<r<1$. 


While the value of the hyperbolic metric can be defined in any plane domain by mapping the domain conformally onto the unit disk \cite[(7.13), p. 133]{kl}, this method does not work for a non-simply connected domain such as the annular ring $R(r,1)$. Consequently, the only way to compute the hyperbolic distance between points $x,y\in R(r,1)$ is to find the infimum of the line integrals of the hyperbolic density in \cite[(7.18), p. 135]{kl} over all rectifiable curves $\gamma$ from $x$ to $y$ in $R(r,1)$, see \cite[Def. 7.3, p. 125]{kl}. Since there is no explicit formula for this infimum, the hyperbolic metric is not very well-suited for measuring the intrinsic distances in this kind of domain.

Thus, in order to study the intrinsic geometry of the ring $R(r,1)$, we use here a few known generalizations of the hyperbolic metric: the $j^*$-metric found first in \cite{hvz} by modifying the distance ratio metric introduced in 1979 by Gehring and Palka \cite{GO79}, the triangular ratio metric introduced by P. H\"ast\"o in 2002 \cite{h} and recently studied in \cite{chkv, sch, fss, sinb, sqm}, and the M\"obius metric originally introduced in \cite[pp. 115-116]{cgqm} and extensively studied by P. Seittenranta in his PhD thesis \cite{S99}. These metrics are important tools in hyperbolic geometry because they share several properties of the hyperbolic metric, such as invariance under similarity mappings, monotonicity with respect to domain and sensitivity to boundary variation \cite[pp. 191-192, 209]{hkv}. Furthermore, as noted in the results of this article, the values of these three metrics can be often bounded with their distances found by rotating the original points around their midpoint.

The structure of this article is as follows. In Section 3, we study the triangular ratio metric in an annular ring and show how the Euclidean midpoint rotation can be used to create upper and lower bounds for the triangular ratio metric in this domain, see Definition \ref{def_emr} and Theorem \ref{thm_emrForS}. In Section 4, we introduce Theorem \ref{thm_deltaInRing} that can be used to compute the M\"obius metric in an annular ring, and also present sharp inequalities between the three hyperbolic type metrics considered. We also inspect the metrical circles drawn with these three metrics. Finally, in Section 5, we consider the M\"obius metric in more general ring domains instead of just an annular ring and present a M\"obius invariant lower bound for the capacity of a ring.

{\bf Acknowledgements.} This research continues my earlier work in \cite{sch, fss, seit, sinb, sqm}. Three last mentioned of these articles have been co-written with Professor Matti Vuorinen, to whom I am indebted for all guidance and support. My research is funded by the University of Turku Graduate School UTUGS. 

\section{Preliminaries}

For any three points $x,y,z\in\R^n$, let $\measuredangle XZY$ be the angle centered at $z$ with the point $x$ on its one side and the point $y$ on its other side. Denote the Euclidean line passing through $x,y$ by $L(x,y)$, the Euclidean line segment from $x$ to $y$ by $[x,y]$, and the smaller angle between the lines $L(x,0)$ and $L(y,0)$ by $\measuredangle XOY$. Furthermore, for all $x\in\R^n$ and $r>0$, let $B^n(x,r)$ be the $x$-centered Euclidean open ball with the radius $r$, $\overline{B}^n(x,r)$ its closure and $S^{n-1}(x,r)$ its boundary sphere. For the unit ball and unit sphere, use the simplified notations $\B^n=B^n(0,1)$ and $S^{n-1}=S^{n-1}(0,1)$. Denote $R(r,1)=\{z\in\C\text{ }|\text{ }r<|z|<1\}$ for $0<r<1$ as in Introduction, and $\overline{\R}^n=\R^n\cup\{\infty\}$.

Define the hyperbolic metric $\rho$ for all points $x,y$ in the unit ball $\B^n$ with the formulas \cite[(4.14), p. 55]{hkv}
\begin{align}\label{formula_rhoB}
\text{sh}^2\frac{\rho_{\B^n}(x,y)}{2}&=\frac{|x-y|^2}{(1-|x|^2)(1-|y|^2)},
\quad
\text{th}\frac{\rho_{\B^2}(x,y)}{2}=\left|\frac{x-y}{1-x\overline{y}}\right|,
\end{align}
where $\overline{y}$ is the complex conjugate of $y$. The hyperbolic metric is \emph{conformally invariant}: If a conformal mapping $f$ fulfills $h:G\to G'=h(G)$ for some domains $G,G'\subset\overline{\R}^n$, then
\begin{align*}
\rho_G(x,y)=\rho_{G'}(h(x),h(y))\quad x,y\in G.    
\end{align*}
According to the Riemann mapping theorem, any simply connected plane domain $G$ can be mapped conformally onto the unit disk, so the aforementioned formula for $\rho_{\B^2}(x,y)$ can be used to compute the hyperbolic metric in different plane domains, such as the upper half-plane and an open sector \cite[Ex. 1, p. 133]{kl}.

By denoting the Euclidean distance from a point $x$ in a domain $G\subsetneq\R^n$ to the boundary $\partial G$ by $d_G(x)=\inf\{|x-z|\text{ }|\text{ }z\in\partial G\}$, we can define the following hyperbolic type metrics:
The \emph{distance ratio metric} \cite[p. 685]{chkv} $j_G:G\times G\to[0,\infty)$,
\begin{align*}
j_G(x,y)=\log\left(1+\frac{|x-y|}{\min\{d_G(x),d_G(y)\}}\right),   
\end{align*}
the \emph{$j^*$-metric} \cite[2.2, p. 1123 \& Lemma 2.1, p. 1124]{hvz} $j^*_G:G\times G\to[0,1],$
\begin{align*}
j^*_G(x,y)={\rm th}\frac{j_G(x,y)}{2}=\frac{|x-y|}{|x-y|+2\min\{d_G(x),d_G(y)\}},    
\end{align*}
the \emph{triangular ratio metric} \cite[(1.1), p. 683]{chkv} $s_G:G\times G\to[0,1],$ 
\begin{align*}
s_G(x,y)=\frac{|x-y|}{\inf_{z\in\partial G}(|x-z|+|z-y|)}. 
\end{align*}

For all distinct points $x,y\in\overline{\R}^n$, define the \emph{spherical (chordal) metric} \cite[(3.6), p. 29]{hkv}
\begin{align*}
q(x,y)=\frac{|x-y|}{\sqrt{1+|x|^2}\sqrt{1+|y|^2}},\quad\text{if}\quad x,y\in\R^n;
\quad
q(x,\infty)=\frac{1}{\sqrt{1+|x|^2}}.
\end{align*}
Using this definition, the expression of the \emph{cross-ratio} can be written for any four distinct points $a,b,c,d\in\overline{\R}^n$ as in \cite[(3.10), p. 33]{hkv}:
\begin{align*}
|a,b,c,d|=\frac{q(a,c)q(b,d)}{q(a,b)q(c,d)},\text{ } a,b,c,d\in\overline{\R}^n;\quad
|a,b,c,d|=\frac{|a-c||b-d|}{|a-b||c-d|},\text{ }a,b,c,d\in\R^n.
\end{align*}
Suppose then $G$ is a domain in $\overline{\R}^n$ so that its complement $(\overline{\R}^n\backslash G)$ contains at least two points. Then the \emph{M\"obius metric} in this domain $G$ is the function $\delta_G:G\times G\to[0,\infty),$ \cite[Def. 1.1, p. 511]{S99}
\begin{align*}
\delta_G(x,y)=\sup_{a,b\in\partial G}\log(1+|a,x,b,y|).    
\end{align*}

While the M\"obius metric is not conformally invariant like the hyperbolic metric, it is invariant under an important subclass of conformal mappings called the M\"obius transformations:

\begin{definition}
\cite[Ex. 3.2, pp. 25-26; Def. 3.6, p. 27 \& Def. 3.7, p. 27]{hkv}
The \emph{hyperplane} perpendicular to a vector $u\in\R^n\backslash\{0\}$ and at distance $t\slash|u|$ from the origin for some $t\geq0$ is 
\begin{align*}
P(u,t)=\{x\in\R^n\text{ }|\text{ }x\cdot u=t\}\cup\{\infty\},   
\end{align*}
where $\cdot$ is the symbol of the dot product. The reflection in a hyperplane $P(u,t)$ is $h_r:\overline{\R}^n\to\overline{\R}^n$,
\begin{align*}
h_r(x)=x-2(x\cdot u-t)\frac{u}{|u|^2},\quad h_r(\infty)=\infty,    
\end{align*}
and the inversion in the sphere $S^{n-1}(v,r)$ is $h_i:\overline{\R}^n\to\overline{\R}^n$,
\begin{align*}
h_i(x)=v+\frac{r^2(x-v)}{|x-v|^2},\quad h_i(v)=\infty,\quad h_i(\infty)=v. \end{align*}
Any function $f:\overline{\R}^n\to\overline{\R}^n$ created as a function composition $f=h_1\circ\cdots\circ h_m$ by combining a number $m\in\Z^+$ of these reflections and inversions is a \emph{M\"obius transformation}. If the number $m$ here is even, the M\"obius transformation $f$ is \emph{sense-preserving}, and otherwise $f$ is \emph{sense-reversing}.
\end{definition}

\begin{example}\label{ex_mobf}
For any $0<r<1$, the function $f:\overline{\R}^2\to\overline{\R}^2$, \begin{align*}
f(x)=rx\slash |x|^2,\quad f(0)=\infty,\quad f(\infty)=0,    
\end{align*}
is an inversion in the circle $S^1(0,\sqrt{r})$ and therefore a sense-reversing M\"obius transformation that preserves the annular ring $R(r,1)$ but maps $S^1(0,r)$ onto $S^1$ and vice versa.
\end{example}

\begin{theorem}\label{thm_seitminvariant} \cite{S99}, \cite[Thm 5.16, p. 75]{hkv}
The M\"obius metric $\delta_G$ is M\"obius invariant: If $G\subset\overline{\R}^n$ is a domain such that ${\rm card}(\overline{\R}^n\backslash G)\geq2$ and $f:\overline{\R}^n\to\overline{\R}^n$ is a M\"obius transformation, then for all $x,y\in G$,
\begin{align*}
\delta_G(x,y)=\delta_{f(G)}(f(x),f(y)).    
\end{align*}
\end{theorem}

Consider the following equalities and inequalities between the metrics introduced above.

\begin{theorem}\label{thm_pr_ineqs}\cite[Thm 5.16, p. 75]{hkv}, \cite[Cor. 3.8, p.5]{seit}
For all points $x,y$ in a domain $G\subset\R^n$ such that ${\rm card}(\overline{\R}^n\backslash G)\geq2$:
\begin{align*}
&(1)\quad
j_G(x,y)\leq\delta_G(x,y)\leq2j_G(x,y),
\quad\text{and}\quad
\delta_G(x,y)=j_G(x,y)
\quad\text{if}\quad
G=\R^n\backslash\{0\},\\
&(2)\quad
j^*_G(x,y)\leq{\rm th}(\delta_G(x,y)\slash2)\leq2j^*_G(x,y),\\
&(3)\quad
s_G(x,y)\slash2\leq{\rm th}(\delta_G(x,y)\slash2)\leq2s_G(x,y),\\
&(4)\quad
\delta_G(x,y)=\rho_G(x,y)
\quad\text{if}\quad
G=\B^n.
\end{align*}
\end{theorem}

\section{Triangular ratio metric in the annular ring}

In this section, we will study the triangular ratio metric in an annular ring $R(r,1)$ with $0<r<1$. In order to do this, we need to find a point $z$ from a certain circle that minimizes the sum $|x-z|+|z-y|$, when both the points $x,y$ are either inside or outside of the circle. This optimization problem has been much studied during its long history because the correct solution $z$ is also the point in which a light ray from the point $x$ must strike a spherical mirror to be reflected to the point $y$, see \cite{fhmv} and \cite[Rmk 1.4, p. 3]{fmv}. The next theorem follows from the law of reflection:

\begin{lemma}\label{lem_zbisectsXZY}
\emph{\cite[Rmk 2.2, p. 137 \& Rmk 2.7, 143]{fhmv}} 
Suppose that there are distinct points $x,y\in\R^n$ and the point $z\in S^{n-1}$ is chosen so that it gives the infimum $\inf_{z\in S^{n-1}}(|x-z|+|z-y|)$. If $x,y\in\B^n$, the line $L(0,z)$ bisects the angle $\measuredangle XZY$, see Figure \ref{fig0}. If $x,y\in\R^n\backslash\overline{\B}^n$ so that $[x,y]\cap\overline{\B}^n\neq\varnothing$, then $z\in[x,y]\cap S^{n-1}$. In the third case where $x,y\in\R^n\backslash\overline{\B}^n$ with $[x,y]\cap\overline{\B}^n=\varnothing$, the line $L(0,z)$ bisects the angle $\measuredangle XZY$, just like in the first case.
\end{lemma}

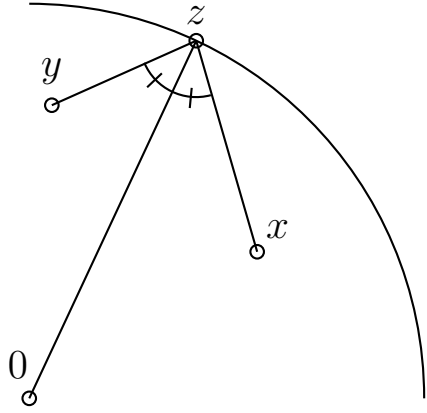
\begin{figure}[ht]
    \centering
    \begin{tikzpicture}[scale=3]
    \draw[thick] (1.75,0) arc (0:90:1.75cm);
    \draw[thick] (0,0) -- (0.74,1.586);
    \draw[thick] (0.1,1.3) -- (0.74,1.586);
    \draw[thick] (1.01,0.65) -- (0.74,1.586);
    \draw[rotate around={203:(0.74,1.586)}, thick] (0.74+0.25,1.586) arc (0:82:0.25cm);
    \draw[rotate around={203+20.5:(0.74,1.586)}, thick] (0.74+0.21,1.586) -- (0.74+0.3,1.586);
    \draw[rotate around={203+61.5:(0.74,1.586)}, thick] (0.74+0.21,1.586) -- (0.74+0.3,1.586);
    \draw[thick] (0,0) circle (0.03cm);
    \draw[thick] (1.01,0.65) circle (0.03cm);
    \draw[thick] (0.1,1.3) circle (0.03cm);
    \draw[thick] (0.74,1.586) circle (0.03cm);
    \node[scale=1.3] at (-0.05,0.15) {$0$};
    \node[scale=1.3] at (1.1,0.75) {$x$};
    \node[scale=1.3] at (0.1,1.45) {$y$};
    \node[scale=1.3] at (0.74,1.7) {$z$};
    \end{tikzpicture}
    \caption{If the point $z$ gives the infimum $\inf_{z\in S^1}(|x-z|+|z-y|)$ for $x,y\in\B^2$, then the line $L(0,z)$ bisects the angle $\measuredangle XZY$.}
    \label{fig0}
\end{figure}

While there is no explicit formula for the point $z$ defining the infimum $\inf_{z\in\partial R(r,1)}(|x-z|+|z-y|)$, it can be solved from the quartic equation presented below.

\begin{theorem}
Consider the annular ring domain $R(r,1)$ with $0<r<1$. Let $x,y\in R(r,1)$ and choose $z$ from the boundary of $R(r,1)$ so that it gives the infimum $\inf_{z\in\partial R(r,1)}(|x-z|+|z-y|)$. Then $z\in[x,y]\cap S^1(0,r)$ if $[x,y]\cap\overline{B}^2(0,r)\neq\varnothing$, and otherwise $z$ fulfills the equality
\begin{align*}
\overline{x}\overline{y}z^4-j^2(\overline{x}+\overline{y})z^3+j^4(x+y)z-j^4xy=0    
\end{align*}
with either $j=r$ or $j=1$.
\end{theorem}
\begin{proof}
The first part follows trivially from the triangle inequality. If $x,y\in R(r,1)$ so that $[x,y]\cap\overline{B}^2(0,r)=\varnothing$ instead, then by Lemma \ref{lem_zbisectsXZY}, $z$ bisects the angle $\measuredangle XZY$. As in \cite[(2.1), p. 138]{fhmv}, $z$ bisects $\measuredangle XZY$ if and only if
\begin{align*}
&\arg\left(\frac{z-x}{z}\right)=\arg\left(\frac{z}{z-y}\right)
\quad\Leftrightarrow\quad
\arg\left(\frac{z-x}{z}\cdot\frac{z-y}{z}\right)=0\\
&\Leftrightarrow\quad
\frac{(z-x)(z-y)}{z^2}=\frac{(\overline{z}-\overline{x})(\overline{z}-\overline{y})}{\overline{z}^2}
\quad\Leftrightarrow\quad
\overline{z}^2(z-x)(z-y)=z^2(\overline{z}-\overline{x})(\overline{z}-\overline{y})\\
&\Leftrightarrow\quad
\overline{z}^2z^2-(x+y)\overline{z}^2z+xy\overline{z}^2=
\overline{z}^2z^2-(\overline{x}+\overline{y})\overline{z}z^2+\overline{x}\overline{y}z^2\\
&\Leftrightarrow\quad
\overline{x}\overline{y}z^2-|z|^2(\overline{x}+\overline{y})z+|z|^2(x+y)\overline{z}-xy\overline{z}^2=0\\
&\Leftrightarrow\quad
\overline{x}\overline{y}z^4-|z|^2(\overline{x}+\overline{y})z^3+|z|^4(x+y)z-|z|^4xy=0.
\end{align*}
Above, the last equivalences follows from the fact that $\overline{z}z=|z|^2$. Since $|z|=r$ or $|z|=1$, when $z\in\partial R(r,1)$, the result follows.
\end{proof}

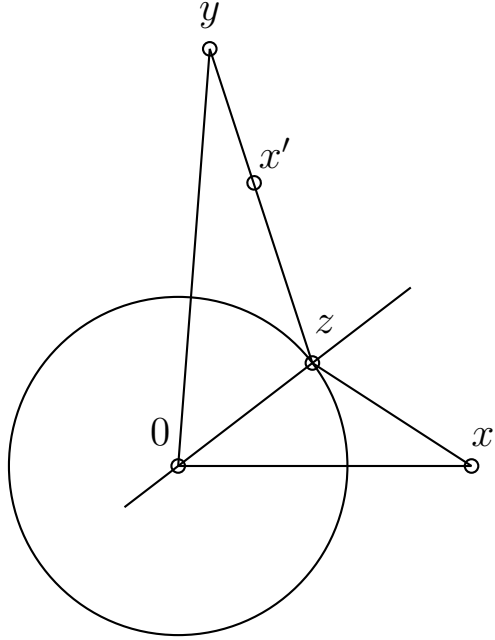
\begin{figure}[ht]
    \centering
    \begin{tikzpicture}[scale=3]
    \draw[thick] (0,0) circle (0.75cm);
    \draw[thick] (0,0) -- (1.3,0);
    \draw[thick] (1.3,0) circle (0.03cm);
    \draw[thick] (0,0) circle (0.03cm);
    \draw[thick] (0,0) -- (0.14,1.85);
    \draw[thick, rotate=75] (1.3,0) circle (0.03cm);
    \draw[thick] (0.14,1.85) circle (0.03cm);
    \draw[thick] (-0.3*0.793,-0.3*0.609) -- (1.3*0.793,1.3*0.609);
    \draw[thick] (0.75*0.793,0.75*0.609) circle (0.03cm);
    \draw[thick] (1.3,0) -- (0.75*0.793,0.75*0.609) -- (0.14,1.85);
    \node[scale=1.3] at (-0.08,0.15) {$0$};
    \node[scale=1.3] at (1.35,0.13) {$x$};
    \node[scale=1.3] at (0.14,2) {$y$};
    \node[scale=1.3] at (0.65,0.63) {$z$};
    \node[scale=1.3] at (0.43,1.39) {$x'$};
    \end{tikzpicture}
    \caption{The line $L(0,z)$ bisects the angle $\measuredangle XZY$, if the point $y$ is on the line $L(z,x')$, where $x'$ is the point $x$ reflected over $L(0,z)$.}
    \label{fig1}
\end{figure}

If the points $x$ and $z$ are fixed instead, finding a point $y$ such that $z$ gives the infimum $\inf_{z\in\partial S^1}(|x-z|+|z-y|)$ is a very simple task and this is often useful when generating points with certain values of the triangular ratio metric.

\begin{proposition}
If $r>0$, $x,y\in\R^2$ with $[x,y]\cap B^2(0,r)=\varnothing$ and the line $L(0,z)$ through the point $z=re^{ui}\in S^1(0,r)$ bisects the angle $\measuredangle XZY$, then there is some $R>0$ such that $y=e^{ui}(R|x|e^{(2u-\mu)i}+1-R)$, where $\mu=\arg(x)$. 
\end{proposition}
\begin{proof}
Since $x=|x|e^{\mu i}$ and $z=e^{ui}$, the point $x$ reflected over the line $L(0,z)$ is $x'=|x|e^{(2u-\mu)i}$. Now, the point $z$ bisects the angle $\measuredangle XZY$ if and only if $y\in L(z,x')$ so that $y$ is on the same side of the line $L(0,z)$ than $x'$, see Figure \ref{fig1}. Consequently, there must be some $R>0$ such that $y=R(x'-z)+z$, from which the result follows.
\end{proof}

While finding the value of the triangular ratio distance between points $x,y\in R(r,1)$ in the general case is quite difficult, there are explicit formulas for this metric if the points $x,y$ are either collinear with the origin or at the same distance from the origin.  

\begin{proposition}\label{prop_sRingColli}
If $x,y\in R(r,1)$ so that $\arg(x)=\arg(y)$, then 
\begin{align*}
s_{R(r,1)}(x,y)=\frac{|x-y|}{\min\{2-|x+y|,|x+y|-2r\}}.    
\end{align*}
\end{proposition}
\begin{proof}
Without loss of generality, we can fix $x,y\in[0,1]\cap R(r,1)$. By Lemma \ref{lem_zbisectsXZY}, $z=r$ or $z=1$. Furthermore, the sum $|x-z|+|z-y|$ is now $2|z-|x+y|\slash2|=|2z-|x+y||$, from which the result follows.
\end{proof}

\begin{theorem}\label{thm_sForConjugate}
\emph{\cite[Thm 3.1, p. 276]{hkvz}} If $x=h+ki\in\B^2$ with $h,k>0$, then
\begin{align*}
s_{\B^2}(x,\overline{x})&=|x|\text{ if }|x-\frac{1}{2}|>\frac{1}{2},\\
s_{\B^2}(x,\overline{x})&=\frac{k}{\sqrt{(1-h)^2+k^2}}\leq|x|\text{ otherwise.}
\end{align*}
\end{theorem}

\begin{lemma}\label{lem_sRingConjugate}
For all $x,y\in R(r,1)$ such that $|x|=|y|=h$ and $\mu\in(0,\pi)$ is the value of the angle $\measuredangle XOY$,
\begin{align*}
&s_{R(r,1)}(x,y)=1
\quad\text{if}\quad
\cos(\mu\slash2)<\frac{r}{h},\\
&s_{R(r,1)}(x,y)=\max\left\{h,\frac{h\sin(\mu\slash2)}{\sqrt{h^2+r^2-2hr\cos(\mu\slash2)}}\right\}
\quad\text{if}\quad
\frac{r}{h}\leq\cos(\mu\slash2)\leq h,\\
&s_{R(r,1)}(x,y)=\frac{h\sin(\mu\slash2)}{\sqrt{h^2+r^2-2hr\cos(\mu\slash2)}}
\quad\text{if}\quad
\max\left\{\frac{r}{h},h\right\}\leq\cos(\mu\slash2)\leq\frac{1+r}{2h},\\
&s_{R(r,1)}(x,y)=\frac{h\sin(\mu\slash2)}{\sqrt{1+h^2-2h\cos(\mu\slash2)}}
\quad\text{if}\quad
\cos(\mu\slash2)>\max\left\{\frac{1+r}{2h},h\right\}.
\end{align*}
\end{lemma}
\begin{proof}
Fix $x=e^{\mu i\slash2}=h(\cos(\mu\slash2)+\sin(\mu\slash2)i)$ and $y=\overline{x}=e^{-\mu i\slash2}$ without loss of generality. If $h\cos(\mu\slash2)\leq r$, then $[x,y]\cap S^1(0,r)\neq\varnothing$ and trivially $s_{R(r,1)}(x,y)=1$. Suppose below that $\cos(\mu\slash2)<r\slash h$ instead. Clearly,
\begin{align*}
s_{R(0,r)}(x,y)=\max\{s_{\R^2\backslash\overline{B}^2(0,r)}(x,y),s_{\B^2}(x,y)\}.
\end{align*}
By Lemma \ref{lem_zbisectsXZY}, the line $L(0,z)$ through the point $z$ giving the infimum $\inf_{z\in S(0,r)}(|x-z|+|z-y|)$ bisects the angle $\measuredangle XZY$ and, consequently, this point $z$ must be $r$. Thus, by the law of cosines and the half-angle formula of the sine function,
\begin{align}\label{eq_sR_inconju}
s_{\R^2\backslash\overline{B}^2(0,r)}(x,y) 
=\frac{|x-y|}{|x-r|+|r-y|}
=\frac{h|1-e^{\mu i}|}{2|he^{\mu i\slash2}-r|}
=\frac{h\sin(\mu\slash2)}{\sqrt{h^2+r^2-2hr\cos(\mu\slash2)}}.
\end{align}
By Theorem \ref{thm_sForConjugate}, $s_{\B^2}(x,y)=h$, if
\begin{align*}
&|x-\frac{1}{2}|>\frac{1}{2}
\quad\Leftrightarrow\quad
|2x-1|=\sqrt{(2h\cos(\mu\slash2)-1)^2+4h^2\sin^2(\mu\slash2)}>1\\
&\Leftrightarrow\quad
4h^2-4h\cos(\mu\slash2)+1>1
\quad\Leftrightarrow\quad
h>\cos(\mu\slash2),
\end{align*}
and otherwise
\begin{align}\label{eq_sB_inconju}
s_{\B^2}(x,y) 
=\frac{h\sin(\mu\slash2)}{\sqrt{1+h^2-2h\cos(\mu\slash2)}}.
\end{align}
The quotient \eqref{eq_sB_inconju} is greater than or equal to \eqref{eq_sR_inconju} if and only if
\begin{align*}
r^2-2hr\cos(\mu\slash2)\geq1-2h\cos(\mu\slash2)
\quad\Leftrightarrow\quad
\cos(\mu\slash2)\geq\frac{1+r}{2h}.
\end{align*}
The lemma follows now by combining all the results above.
\end{proof}

Any distinct points $x,y\in R(r,1)$ can be rotated around their midpoint in the following way so that the value of the triangular ratio metric for the rotated points can be found with either Proposition \ref{prop_sRingColli} or Lemma \ref{lem_sRingConjugate}.

\begin{definition}\label{def_emr}
\cite[Def. 4.1, p. 10]{sinb}
\emph{Euclidean midpoint rotation.} Choose distinct point $x,y\in\R^2$. Denote $k=(x+y)\slash2$ and $q=|x-k|=|y-k|$. Fix then four distinct points $x_0,y_0,x_1,y_1\in S^1(k,q)$ so that $(x_0+y_0)\slash2=(x_1+y_1)\slash2=k$, $|x_0|=|y_0|$, $|x_1|=|k|+q$ and $|y_1|=|k|-q$, see Figure \ref{fig3}. Note that if $x,y\in R(r,1)$, then $x_0,y_0\in R(r,1)$ always, but it might be so that $y_1\in\overline{B}^2(0,r)$ or $x_1\notin\B^2$.
\end{definition}

\begin{figure}[ht]
    \centering
    \begin{tikzpicture}[scale=3]
    \draw[thick] (2.7,0) arc (0:90:2.7);
    \draw[thick] (1,0) arc (0:90:1);
    \draw[thick] (1.3,1.3) circle (0.6cm);
    \draw[rotate around={-15:(1.3,1.3)}, thick, dashed] (1.3,1.9) -- (1.3,0.7);
    \draw[rotate around={-45:(1.3,1.3)}, thick, dashed] (1.3,1.9) -- (1.3,-0.55);
    \draw[rotate around={45:(1.3,1.3)}, thick, dashed] (1.3,1.9) -- (1.3,0.7);
    \draw[rotate around={-15:(1.3,1.3)}, thick] (1.3,1.9) circle (0.03cm);
    \draw[rotate around={-45:(1.3,1.3)}, thick] (1.3,1.9) circle (0.03cm);
    \draw[rotate around={45:(1.3,1.3)}, thick] (1.3,1.9) circle (0.03cm);
    \draw[rotate around={165:(1.3,1.3)}, thick] (1.3,1.9) circle (0.03cm);
    \draw[rotate around={135:(1.3,1.3)}, thick] (1.3,1.9) circle (0.03cm);
    \draw[rotate around={225:(1.3,1.3)}, thick] (1.3,1.9) circle (0.03cm);
    \draw[thick] (0,0) circle (0.03cm);
    \draw[thick] (1.3,1.3) circle (0.03cm);
    \node[scale=1.3] at (0,0.15) {$0$};
    \node[scale=1.3] at (1.25,1.5) {$k$};
    \node[scale=1.3] at (1.45,2.03) {$x$};
    \node[scale=1.3] at (1.1,0.57) {$y$};
    \node[scale=1.3] at (0.85,1.87) {$x_0$};
    \node[scale=1.3] at (1.83,0.76) {$y_0$};
    \node[scale=1.3] at (1.77,1.84) {$x_1$};
    \node[scale=1.3] at (0.71,0.87) {$y_1$};
    \end{tikzpicture}
    \caption{Euclidean midpoint rotation for the points $x,y$ in an annular ring $R(r,1)$.}
    \label{fig3}
\end{figure}
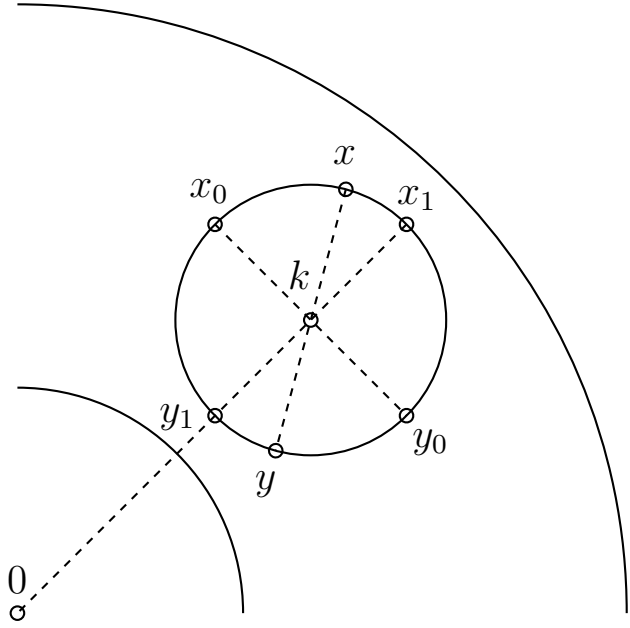

This midpoint rotation can be used to find bounds for the value of the triangular ratio metric, because the rotated points fulfill the inequality of Theorem \ref{thm_emrForS}.

\begin{proposition}\label{prop_functionf} \cite[Prop. 4.3, p. 11]{sinb}
A function $f:[0,\pi\slash2]\to\R$,
\begin{align*}
f(\mu)=\sqrt{u-v\cos(\mu)}+\sqrt{u+v\cos(\mu)},    
\end{align*} 
where $u,v>0$ are constants, is increasing on the interval $\mu\in[0,\pi\slash2]$. 
\end{proposition}

\begin{theorem}\label{thm_diskdiameter} \cite[Cor. 4.6, p. 14]{sinb}
The $s_G$-diameter of a closed ball $\overline{B}^n(k,q)$ in a domain $G\subsetneq\R^n$ is $s_G(\overline{B}^n(k,q))=q\slash(q+d)$, where $d=\inf\{|x-z|\text{ }|\text{ }x\in\overline{B}^n(k,q),\,z\in\partial G\}$.
\end{theorem}

\begin{theorem}\label{thm_emrForS}
For all distinct points $x,y\in R(r,1)$, fix $x_0,y_0,x_1,y_1$ as in Definition \ref{def_emr}. If the distance $s_{R(r,1)}(x_1,y_1)$ is well-defined for $x_1,y_1\in R(r,1)$, the inequality 
\begin{align*}
s_{R(r,1)}(x_0,y_0)\leq s_{R(r,1)}(x,y)\leq s_{R(r,1)}(x_1,y_1)    
\end{align*}
holds. Otherwise only the first part of this inequality holds and the points $x,y$ can be rotated around their Euclidean midpoint into new points $x',y'$ so that $s_{R(r,1)}(x',y')\to1^-$.
\end{theorem}
\begin{proof}
Denote $k=(x+y)\slash2$ and $q=|x-k|$. Suppose without loss of generality that $k\in[0,1]\cap R(r,1)$. Now, the points of Definition \ref{def_emr} can be written as $x_0=k+qi$, $y_0=k-qi$, $x_1=k+q$ and $y_1=k-q$. Furthermore, denote the angle between $L(x,y)$ and the real axis by $\mu$.

It follows from Lemma \ref{lem_zbisectsXZY} that the infimum $\inf_{z\in S(0,r)}(|x_0-z|+|z-y_0|)$ is given by $z=r$ and thus
\begin{align*}
s_{\R^2\backslash\overline{B}^2(0,r)}(x_0,y_0)
=\frac{|x_0-y_0|}{|x_0-r|+|r-y_0|}
=\frac{q}{\sqrt{(k-r)^2+q^2}}.
\end{align*}
By the law of cosines and Proposition \ref{prop_functionf},
\begin{align*}
&s_{\R^2\backslash\overline{B}^2(0,r)}(x,y) 
\geq\frac{|x-y|}{|x-r|+|r-y|}\\
&=\frac{2q}{\sqrt{q^2+(k-r)^2-2q(k-r)\cos(\mu)}+\sqrt{q^2+(k-r)^2+2q(k-r)\cos(\mu)}}\\
&\geq\frac{q}{\sqrt{(k-r)^2+q^2}}
=s_{\R^2\backslash\overline{B}^2(0,r)}(x_0,y_0).
\end{align*}
Since $s_{\B^2}(x,y)\geq s_{\B^2}(x_0,y_0)$ by \cite[Thm 4.11, p. 15]{sinb}, it follows that
\begin{align*}
&s_{R(r,1)}(x_0,y_0)
=\max\{s_{\R^2\backslash\overline{B}^2(0,r)}(x_0,y_0),s_{\B^2}(x_0,y_0)\}
\leq\max\{s_{\R^2\backslash\overline{B}^2(0,r)}(x,y),s_{\B^2}(x,y)\}\\ &=s_{R(r,1)}(x,y),   
\end{align*}
which proves the first part of the inequality.

Suppose next that $x_1,y_1\in R(r,1)$. Now, $x,y\in\overline{B}^2(k,q)\subset R(r,1)$ and, by Theorem \ref{thm_diskdiameter} and Proposition \ref{prop_sRingColli},
\begin{align*}
s_{R(r,1)}(x,y)
&=\max\{s_{\R^2\backslash\overline{B}^2(0,r)}(x,y),s_{\B^2}(x,y)\}
\leq\max\{s_{\R^2\backslash\overline{B}^2(0,r)}(\overline{B}^2(k,q)),s_{\B^2}(\overline{B}^2(k,q))\}\\
&=\max\left\{\frac{q}{q+k-q-r},\frac{q}{q+1-k-q}\right\}
=\max\left\{\frac{q}{k-r},\frac{q}{1-k}\right\}\\
&=\max\left\{\frac{|x_1-y_1|}{|x_1+y_1|-2r},\frac{|x_1-y_1|}{2-|x_1+y_1|}\right\}
=s_{R(r,1)}(x_1,y_1).
\end{align*}
If $x_1\notin R(r,1)$ or $y_1\notin R(r,1)$, then $\overline{B}^2(k,q)\cap(\partial R(r,1))\neq\varnothing$ and the rest of the theorem follows trivially.
\end{proof}

A similar result also holds in the punctured unit disk $\B^2\backslash\{0\}$. 

\begin{lemma}
For all $x,y\in\B^2\backslash\{0\}$, 
\begin{align*}
s_{\B^2\backslash\{0\}}(x_0,y_0)
\leq s_{\B^2\backslash\{0\}}(x,y)
\leq s_{\B^2\backslash\{0\}}(x_1,y_1),
\end{align*}
where $x_0,y_0,x_1,y_0$ are as in Definition \ref{def_emr} (assuming that $y_0\neq0$ and $x_0\in\B^2$ so that $s_{\B^2\backslash\{0\}}(x_1,y_1)$ is well-defined).
\end{lemma}
\begin{proof}
Clearly,
\begin{align*}
s_{\B^2\backslash\{0\}}(x,y)
=\max\left\{s_{\B^2}(x,y),\frac{|x-y|}{|x|+|y|}\right\},
\end{align*}
so the result follows from \cite[Thm 4.4, p. 12 \& Thm 4.11, p. 15]{sinb}.
\end{proof}

\section{M\"obius metric in the annular ring}

In this section, we will study the M\"obius metric defined in the annular ring $R(r,1)$. Our main result is the following theorem that can be used to compute the exact value of the M\"obius metric for any points $x,y$ in $R(r,1)$. Since the supremum in Theorem \ref{thm_deltaInRing} only needs to be found for one variable $v$ defined in a closed real-valued interval $[\mu,\pi]$ with $\mu\in[0,\pi]$ instead of the two boundary points in the general expression of the M\"obius metric, this result can be used to write an effective algorithm for computation of the M\"obius metric with some single-variable optimization function, like the function \emph{optimize} in R.  

\begin{theorem}\label{thm_deltaInRing}
For all $x,y\in R(r,1)$ with $|y|\leq|x|$,
\begin{align*}
\delta_{R(r,1)}(x,y)
=\max\left\{
\rho_{\B^2}(x,y),\,
\rho_{\B^2}\left(\frac{rx}{|x|^2},\frac{ry}{|y|^2}\right),\,
\sup_{v\in[\mu,\pi]}\log(1+|e^{-u(v)i},|x|,re^{vi},|y|e^{\mu i}|)
\right\},
\end{align*}
where $\mu\in(0,\pi)$ is the value of the angle $\measuredangle XOY$, and 
\begin{align*}
&u(v)=\arcsin\left(\frac{-c_2+\sqrt{c_2^2-4c_1c_3}}{2c_1}\right)\quad\text{with}\\
&c_1=|x|^2(1+r^2)^2+r^2(1+|x|^2)^2-2r|x|(1+r^2)(1+|x|^2)\cos(v),\\
&c_2=4r|x|\sin(v)(|x|(1+r^2)-r(1+|x|^2)\cos(v)),\\
&c_3=-r^2\sin(v)^2(1-|x|^2)^2.    
\end{align*}
\end{theorem}
\begin{proof}
Since $|y|\leq|x|$, trivially
\begin{align*}
\sup_{a\in S^1,\, b\in S^1(0,r)}|a,x,b,y|\geq\sup_{a\in S^1(0,r),\, b\in S^1}|a,x,b,y|.    
\end{align*}
Thus,
\begin{align*}
\delta_{R(r,1)}(x,y)
&=\sup_{a,b\in\partial R(r,1)}\log(1+|a,x,b,y|)\\
&=\max\{
\delta_{\B^2}(x,y),\,
\delta_{\overline{\R}^2\backslash\overline{B}^2(0,r)}(x,y),\,
\sup_{a\in S^1,\, b\in S^1(0,r)}\log(1+|a,x,b,y|)
\}.
\end{align*}
By using the M\"obius transformation $f$ from Example \ref{ex_mobf} and the M\"obius invariance of the M\"obius metric, it follows that
\begin{align*}
\delta_{\overline{\R}^2\backslash\overline{B}^2(0,r)}(x,y)
=\delta_{\B^2}(f(x),f(y))
=\delta_{\B^2}(rx\slash |x|^2,ry\slash |y|^2).
\end{align*}
Since $\delta_{\B^2}(x,y)=\rho_{\B^2}(x,y)$ by Theorem \ref{thm_pr_ineqs}(4), the first two components of the maximum expression above are the same as in the theorem.

Let us now maximize the supremum $\sup_{a\in S^1,\, b\in S^1(0,r)}|a,x,b,y|$ with respect to $a$. Since the cross-ratio is invariant under reflections and rotations in the ring $R(r,1)$, we can assume without loss of generality that $\arg(x)=0$ and $0\leq\arg(y)\leq\pi$. Note that if $\arg(a)\in(0,\pi)$ or $\arg(b)\in(\pi,2\pi)$ now, we can reflect one or both of the points $a,b$ over the real axis without decreasing the cross-ratio $|a,x,b,y|$. Thus, we can also suppose that $\arg(a)\in(\{0\}\cup[\pi,2\pi))$ and $\arg(b)\in[0,\pi]$, and fix $a=e^{-ui}$ and $b=re^{vi}$ with $u,v\in[0,\pi]$. By the law of cosines,
\begin{align*}
|a,x,b,y|
=\frac{|x-y||a-b|}{|x-a||b-y|}
=\frac{|x-y||e^{-ui}-re^{vi}|}{||x|-e^{-ui}||y-b|}
=\frac{|x-y|}{|y-b|}\sqrt{\frac{1+r^2-2r\cos(u+v)}{1+|x|^2-2|x|\cos(u)}}.
\end{align*}
By differentiation and the angle sum formula of the sine function,
\begin{align*}
&\frac{\partial}{\partial u}\left(\frac{1+r^2-2r\cos(u+v)}{1+|x|^2-2|x|\cos(u)}\right)\\
&=\frac{2r\sin(u+v)(1+|x|^2-2|x|\cos(u))-2|x|\sin(u)(1+r^2-2r\cos(u+v))}{(1+|x|^2-2|x|\cos(u))^2}\\
&=\frac{2(r(1+|x|^2)\sin(u+v)-|x|(1+r^2)\sin(u)-2r|x|\sin(v))}{(1+|x|^2-2|x|\cos(u))^2}.
\end{align*}
Denote $t=\sin(u)$. By the angle sum formula and the quadratic formula,
\begin{align*}
&r(1+|x|^2)\sin(u+v)-|x|(1+r^2)t-2r|x|\sin(v)=0\\
\quad\Leftrightarrow\quad
&r(1+|x|^2)\sin(v)\cos(u)=
(|x|(1+r^2)-r(1+|x|^2)\cos(v))t+2r|x|\sin(v)\\
\quad\Leftrightarrow\quad
&r^2(1+|x|^2)^2\sin^2(v)(1-t^2)
=((|x|(1+r^2)-r(1+|x|^2)\cos(v))t+2r|x|\sin(v))^2\\
\quad\Leftrightarrow\quad
&t=\frac{-c_2\pm\sqrt{c_2^2-4c_1c_3}}{2c_1}\quad\text{with}\\
&c_1=|x|^2(1+r^2)^2+r^2(1+|x|^2)^2-2r|x|(1+r^2)(1+|x|^2)\cos(v),\\
&c_1=4r|x|\sin(v)(|x|(1+r^2)-r(1+|x|^2)\cos(v)),\\
&c_3=-r^2\sin(v)^2(1-|x|^2)^2.
\end{align*}
Here, $0\leq t\leq1$ because $0\leq u\leq\pi$ and only the positive root for $t$ is a zero of the derivative above. It can be verified that the cross-ratio $|a,x,b,y|$ is majorized by $|e^{-ui},x,b,y|$ for $u$ such that $t=\sin(u)$ is the positive root in question. Clearly, this solution fulfills $0\leq u\leq\pi\slash2$ because otherwise $\cos(u)<0$ and the left side of the second equality above is negative, unlike the right side. Thus, we can choose $u=\arcsin(t)$, where $t$ is the positive root above. Since now $\arg(x)=0$, $0\leq\arg(y)\leq\pi$ and $\arg(a)\in(\{0\}\cup[3\pi\slash2,2\pi))$, we see that $\arg(b)$ must be on the interval $[\arg(y),\pi]$: If $0\leq\arg(b)\leq\arg(y)$, we could reflect $b$ over the line $L(0,y)$ to increase the cross-ratio $|a,x,b,y|$. Thus, the result follows.
\end{proof}

With Theorem \ref{thm_deltaInRing}, we can find the value of the M\"obius metric for points collinear with the origin in the annular ring $R(r,1)$.

\begin{corollary}\label{cor_deltaRingForCollinear}
For all points $x,y\in R(r,1)$ collinear with the origin such that $|y|\leq|x|$,
\begin{align*}
{\rm th}\frac{\delta_{R(r,1)}(x,y)}{2}
=\max\left\{
\frac{|x|-|y|}{1-|x||y|},\,
\frac{r(|x|-|y|)}{|x||y|-r^2},\,
\frac{(|x|-|y|)(1-r)}{2(1-|x|)(|y|-r)+(|x|-|y|)(1-r)}
\right\},
\end{align*}
if the value of the angle $\measuredangle XOY$ is 0, and
\begin{align*}
{\rm th}\frac{\delta_{R(r,1)}(x,y)}{2}
=\max\left\{
\frac{|x|+|y|}{1+|x||y|},\,
\frac{r(|x|+|y|)}{|x||y|+r^2},\,
\frac{(|x|+|y|)(1+r)}{2(1-|x|)(|y|-r)+(|x|+|y|)(1+r)}
\right\},
\end{align*}
if the value of the angle $\measuredangle XOY$ is $\pi$.
\end{corollary}
\begin{proof}
Suppose without loss of generality that $x\in R(r,1)\cap[0,1]$. Consider first the case where $y$ is on the line segment $[0,1]$, too. It can be shown by differentiation that $(k+u+v)\slash(ku)$ with constants $u,v>0$ is decreasing with respect to $k>0$. Consequently, for all $a\in S^1$ and $b\in S^1(0,r)$,
\begin{align*}
\frac{|a-b|}{|x-a||y-b|}
\leq\frac{|x-a|+|x-y|+|y-b|}{|x-a||y-b|}
\leq\frac{|x-1|+|x-y|+|y-r|}{|x-1||y-r|}
=\frac{|1-r|}{|x-1||y-r|}
\end{align*}
and therefore $\sup_{a\in S^1,\,b\in S^1(0,r)}|a,x,b,y|=|1,x,r,y|$ now. If $y\in R(r,1)\cap[0,-1]$ instead, then by Theorem \ref{thm_deltaInRing},
\begin{align*}
\sup_{a\in S^1,\,b\in S^1(0,r)}|a,x,b,y|
=\sup_{v\in[\mu,\pi]}|e^{-u(v)i},|x|,re^{vi},|y|e^{\mu i}|
=|e^{-u(\pi)i},x,re^{\pi i},y|
=|1,x,-r,y|,
\end{align*}
because the value of the angle $\measuredangle XOY$ is $\mu=\pi$ and therefore $v\in[\mu,\pi]=\{\pi\}$ here. The result now follows from Theorem \ref{thm_deltaInRing}, the formula \eqref{formula_rhoB} and the observation that
\begin{align*}
{\rm th}\frac{\log(1+k\slash h)}{2}
=\frac{k}{2h+k}
\quad\text{for}\quad k,h>0.
\end{align*}
\end{proof}

Next, we will study the Euclidean midpoint rotation for the M\"obius metric but, first, consider the following conjecture.

\begin{conjecture}\label{conj_decCr}
For all $r<k<1$ and $0<q<\min\{k-r,1-k\}$, the supremum 
\begin{align*}
\sup_{a\in S^1,\,b\in S^1(0,r)}|a,qe^{\mu i}+k,b,qe^{(\pi+\mu)i}+k|  
\end{align*}
is decreasing with respect to $\mu\in[0,\pi\slash2]$.
\end{conjecture}

If Conjecture \ref{conj_decCr} holds, the M\"obius metric fulfills the same inequality as the triangular ratio metric in Theorem \ref{thm_emrForS} and we can use the results of Corollary \ref{conj_decCr} to create upper bounds for the M\"obius metric defined in the annular ring $R(r,1)$.

\begin{lemma}\label{lem_emrForDelta}
Choose distinct points $x,y\in R(r,1)$ so that $|y|\leq|x|$ and $x_0,y_0,x_1,y_1\in R(r,1)$, when these points are as in Definition \ref{def_emr}. Suppose that Conjecture \ref{conj_decCr} holds. Now, the M\"obius metric fulfills
\begin{align*}
\delta_{R(r,1)}(x_0,y_0)
\leq\delta_{R(r,1)}(x,y)
\leq\delta_{R(r,1)}(x_1,y_1).
\end{align*}
\end{lemma}
\begin{proof}
Denote $k=(x+y)\slash2$ and $q=|x-y|\slash2$. Suppose without loss of generality that $k\in(1,r)$ and $x=qe^{\mu i}+k$, $y=qe^{(\pi+\mu)i}+k$ for some $0\leq\mu\leq\pi\slash2$. Now, $(x,y)=(x_0,y_0)$ if $\mu=\pi\slash2$ and $(x,y)=(x_1,y_1)$ if $\mu=0$. Recall from Theorem \ref{thm_deltaInRing} and its proof that the expression of $\delta_{R(r,1)}$ can be written as a maximum expression with three components:
\begin{align*}
\delta_{R(r,1)}=\max\{
\rho_{\B^2}(x,y),
\rho_{\overline{\R}^2\backslash\overline{B}^2(0,r)}(x,y), 
\sup_{a\in S^1,\,b\in S^1(0,r)}\log(1+|a,x,b,y|)
\}.
\end{align*}
The result follows if each of these three components is decreasing with respect to $\mu$ and we already assumed that the cross-ratio fulfills this condition. Furthermore, we can write
\begin{align*}
{\rm th}\frac{\rho_{\B^2}(x,y)}{2}
&=\left|\frac{x-y}{1-x\overline{y}}\right|
=\frac{2q}{|1-(qe^{\mu i}+k)(qe^{(\pi+\mu)i}+k)|}\\
&=\frac{2q}{|1+q^2-k^2-2kq\sin(\mu)i|}
=\frac{2q}{\sqrt{(1+q^2-k^2)^2+4k^2q^2\sin^2(\mu)}},
\end{align*}
and this is clearly decreasing with respect to $\mu\in[0,\pi\slash2]$. Suppose next that $[x,y]$ is some diameter of $S^1(k^*,q^*)\subset\R^2\backslash\overline{\B}^2$ where $q^*,k^*>0$. This circle can be mapped with an inversion in the unit disk onto another circle inside the unit disk, and both the hyperbolic metric and angle magnitudes are invariant under inversions. Thus, it follows directly from above that $\rho_{\overline{\R}^2\backslash\overline{\B}^2}(q^*e^{\mu i}+k^*,q^*e^{(\pi+\mu)i}+k^*)$ is decreasing with respect to $\mu\in[0,\pi\slash2]$ and, since this property is invariant under the stretching by a factor $r>0$, $\rho_{\overline{\R}^2\backslash\overline{B}^2(0,r)}(q^*e^{\mu i}+k^*,q^*e^{(\pi+\mu)i}+k^*)$ is decreasing with respect to $\mu$, too.
\end{proof}

\begin{remark}
Also the $j^*$-metric fulfills the inequality related to the Euclidean midpoint rotation, as can be trivially seen from its definition in \eqref{jinRing}.
\end{remark}

\begin{figure}[ht!]%
 \centering
 \hspace{1cm}
 \subfloat[]{\includegraphics[scale=0.65]{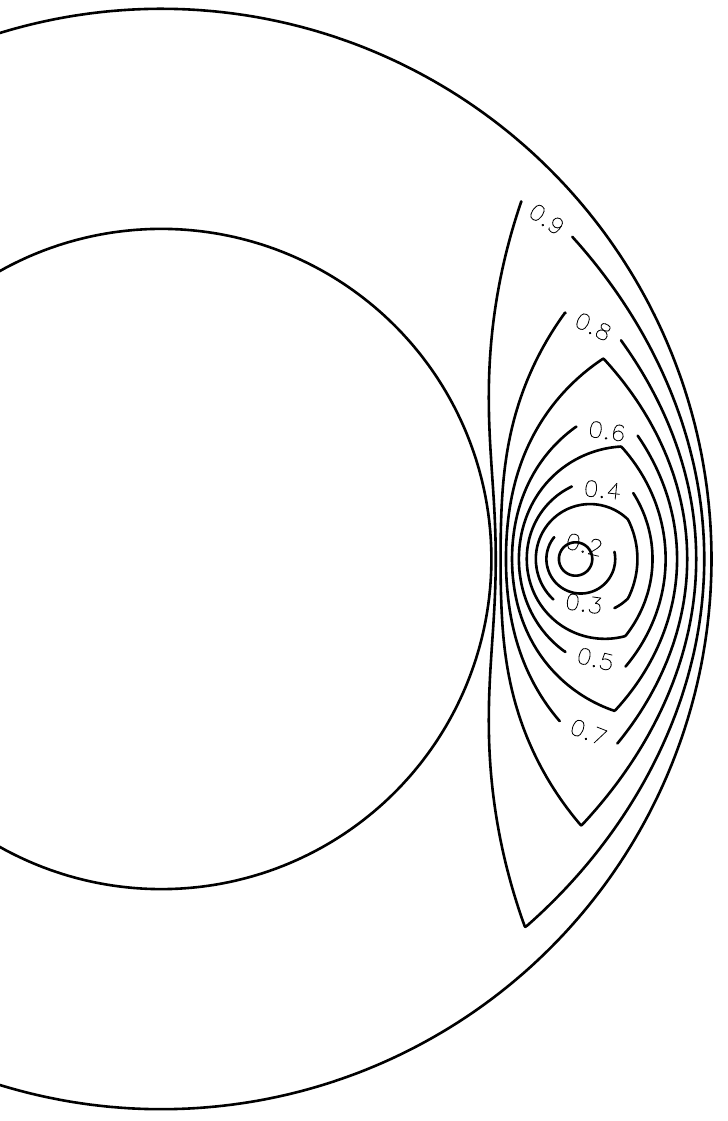}\label{figu1}}\hspace{1.3cm}%
 \subfloat[]{\includegraphics[scale=0.65]{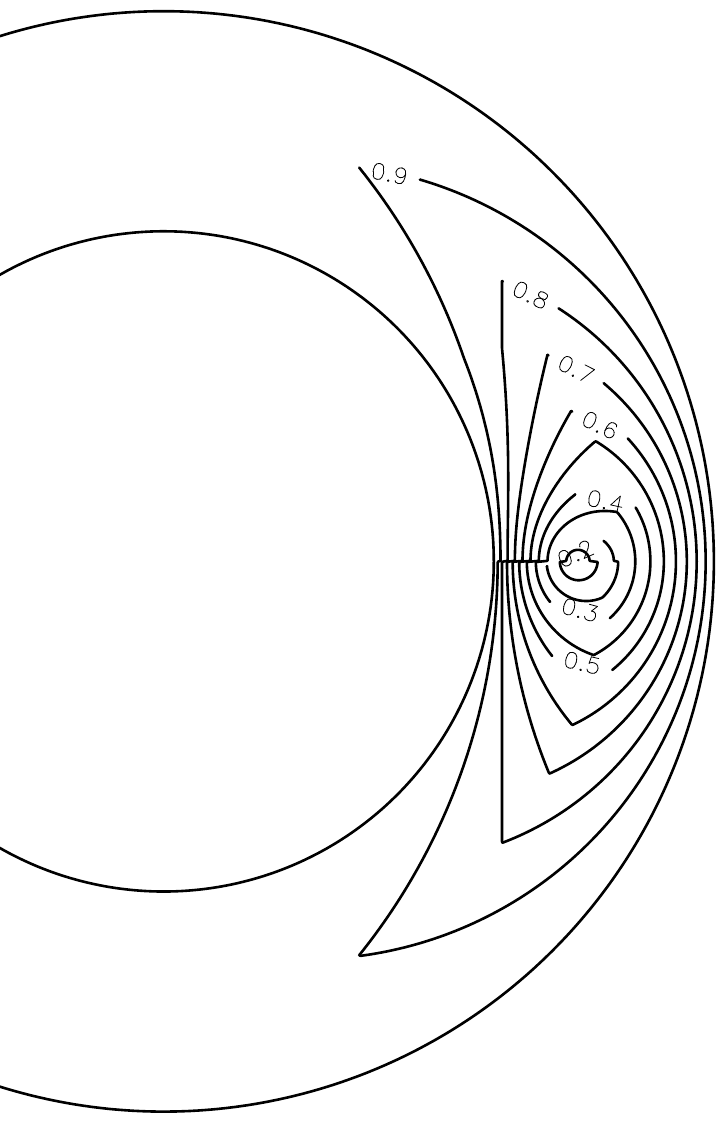}\label{figd1}}\\
 \subfloat[]{\includegraphics[scale=0.65]{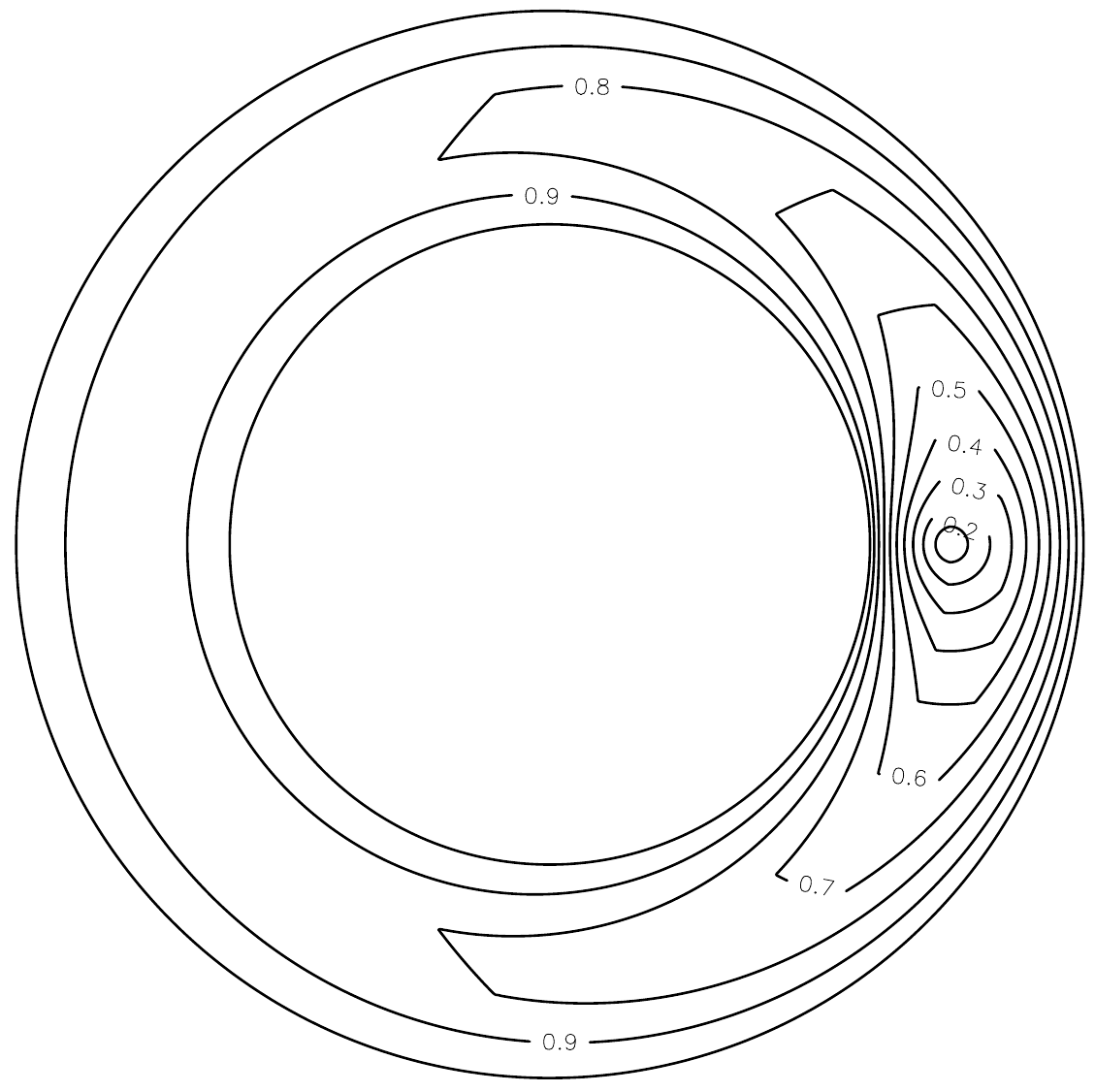}\label{figj1}}%
 \caption{Circles $\{y\in R(r,1)\text{ }|\text{ }d(x,y)=\ell\}$ drawn with the triangular ratio metric $s_{R(r,1)}(x,y)$ (A), the modification ${\rm th}(\delta_{R(r,1)}(x,y)\slash2)$ of the M\"obius metric (B), and the $j^*$-metric $j^*_{R(r,1)}(x,y)$ (C), when $x=0.75$, $r=0.6$ and $\ell=0.2,0.3,...,0.9$}%
 \label{figc}%
\end{figure}

Next, let us compare the three different hyperbolic type metrics in an annular ring $R(r,1)$. The different properties of the triangular ratio metric, the M\"obius metric and the $j^*$-metric can be visually demonstrated by plotting the metrical circles
\begin{align}\label{circs}
\begin{split}
S_s(x,\ell)&=\{y\in R(r,1)\text{ }|\text{ }s_{R(r,1)}(x,y)=\ell\},\\
S_{\delta^*}(x,\ell)&=\{y\in R(r,1)\text{ }|\text{ }{\rm th}(\delta_{R(r,1)}(x,y)\slash2)=\ell\},\\
S_{j^*}(x,\ell)&=\{y\in R(r,1)\text{ }|\text{ }j^*_{R(r,1)}(x,y)=\ell\}
\end{split}
\end{align}
for a few different values of $\ell\in(0,1)$. Note that the circles of the M\"obius metric are drawn by fixing their radius to be the hyperbolic tangent function with the value of the M\"obius metric divided by two as an input, so that these circles are comparable to those of the other two metrics which can only attain values from the interval $[0,1]$.   

Figure \ref{figc} shows these metric circles in \eqref{circs}, when their center is fixed to $x=0.75$, the inner radius of the annular ring is $r=0.6$ and the radii of the metrical circles are $\ell=0.2,0.3,...,0.9$. All these figures where drawn in R-Studio by using a grid of the size 1,000$\times$1,000 test points and the contourplot function \emph{contour}. The values of the triangular ratio metric and the M\"obius metric were computed with the optimization function \emph{optimize} with the help of Theorem \ref{thm_deltaInRing}, and the values of the $j^*$-metric directly obtained with the formula
\begin{align}\label{jinRing}
j^*_{R(r,1)}(x,y)=\frac{|x-y|}{|x-y|+2\min\{|x|-r,|y|-r,1-|x|,1-|y|\}}.    
\end{align}

As we can see from Figure \ref{figc}, the metrical circles drawn with the triangular ratio metric and the M\"obius metric resemble each other more than they do the circles of the $j^*$-metric. Figure \ref{figc}(a) shows that the triangular ratio metric disks are \emph{starlike} with respect to their center $x$: For all points $y\in R(r,1)$ with $s_{R(r,1)}(x,y)<\ell\in(0,1)$, $[x,y]\cap S_s(x,\ell)=\varnothing$. This is a general property of the triangular ratio metric \cite[p. 206]{hkv}, which follows from the fact that $s_G(x,y)=1$ whenever $[x,y]\cap\partial G\neq\varnothing$. We also notice from Figure \ref{figc} that the metric circles drawn with the smallest radii $\ell=0.2$ resemble Euclidean circles while the shape of the domain affects more clearly the circles with larger radii, which is a common property of a hyperbolic type metric, see \cite[Ch. 13, pp. 239-259]{hkv}.

Consider then the following inequality between these three metrics.

\begin{theorem}\label{thm_ringIne}
For all $x,y\in R(r,1)$,
\begin{align*}
\frac{1}{2}s_{R(r,1)}(x,y)
\leq j^*_{R(r,1)}(x,y)
\leq {\rm th}\frac{\delta_{R(r,1)}(x,y)}{2}
\leq 2j^*_{R(r,1)}(x,y)
\leq 2s_{R(r,1)}(x,y),
\end{align*}
where the constants $1\slash2$ and $1$ are sharp when $r\to0^+$, and the constants $2$ are sharp for all values of $r\in(0,1)$.
\end{theorem}
\begin{proof}
By \cite[Lemma 2.1, p. 1124 \& Lemma 2.2, p. 1125]{hvz}, $j^*_G(x,y)\leq s_G(x,y)\leq2j^*_G(x,y)$ holds for all $x,y\in G\subsetneq\R^n$, so the inequality of the theorem follows from this and Lemma \ref{thm_pr_ineqs}(2). By Corollary \ref{cor_deltaRingForCollinear}, for $x=\sqrt{r}$ and $y=-\sqrt{r}$,
\begin{align*}
\frac{{\rm th}(\delta_{R(r,1)}(x,y)\slash2)}{s_{R(r,1)}(x,y)}
=\frac{1+r}{2(1-\sqrt{r}+r)}
\to\frac{1}{2},
\quad
\frac{{\rm th}(\delta_{R(r,1)}(x,y)\slash2)}{j^*_{R(r,1)}(x,y)}
=\frac{(1+r)(2+\sqrt{r})}{2(1-\sqrt{r}+r)}
\to1,
\end{align*}
when $r\to0^+$. Similarly by Corollary \ref{cor_deltaRingForCollinear}, for $x=(1+r)\slash2+h$ and $y=(1+r)\slash2-h$ with $0<r<1$ and $0<h<(1-r)\slash2$,
\begin{align*}
\frac{{\rm th}(\delta_{R(r,1)}(x,y)\slash2)}{j^*_{R(r,1)}(x,y)}
=\frac{{\rm th}(\delta_{R(r,1)}(x,y)\slash2)}{s_{R(r,1)}(x,y)}
=\frac{2(1-r)^2}{(1-r)^2+8h(1-r)+4h^2}
\to2,
\end{align*}
when $h\to0^+$. Thus, the observation about sharpness follows.
\end{proof}

The result of Theorem \ref{thm_ringIne} is useful because it gives us bounds for the distortion of the triangular ratio metric and the $j^*$-metric under M\"obius transformations defined in the annular ring $R(r,1)$. While both these metrics are invariant under rotations about the origin and reflections over lines passing through the origin, they are not M\"obius invariant. For instance, their values clearly change under the inversion $f:x\mapsto rx\slash|x|^2$ of Example \ref{ex_mobf}.  

\begin{corollary}\label{cor_sjUnderMobRing}
For all $x,y\in R(r,1)$ and any M\"obius transformation $f:\overline{\R}^2\to\overline{\R}^2$ such that $f(R(r,1))=R(r,1)$,
\begin{align*}
\frac{1}{4}s_{R(r,1)}(x,y)
\leq s_{R(r,1)}(f(x),f(y))
\leq 4s_{R(r,1)}(x,y),\\
\frac{1}{2}j^*_{R(r,1)}(x,y)
\leq j^*_{R(r,1)}(f(x),f(y))
\leq 2j^*_{R(r,1)}(x,y).
\end{align*}
\end{corollary}
\begin{proof}
Follows from Theorem \ref{thm_ringIne} and the M\"obius invariance of the M\"obius metric. 
\end{proof}

\begin{remark}
Computer tests suggest that the Lipschitz constant $Lip(f|R(r,1))$ of the function $f$ of Example \ref{ex_mobf} is 2 for both the triangular ratio metric and the $j^*$-metric, so probably only the inequality for the distortion of the $j^*$-metric in Corollary \ref{cor_sjUnderMobRing} is sharp. \end{remark}

Let us yet study the M\"obius metric in the punctured unit disk $\B^2\backslash\{0\}$.

\begin{theorem}\label{thm_deltaPuncDisk}
For all $x,y\in\B^2\backslash\{0\}$ with $|y|\leq|x|$,
\begin{align*}
\delta_{\B^2\backslash\{0\}}(x,y)
=\max\left\{
\rho_{\B^2}(x,y),\,
\log\left(1+\frac{|x-y|}{(1-|x|)|y|}\right)
\right\}.
\end{align*}
\end{theorem}
\begin{proof}
Since $|y|\leq|x|$, the point $a$ giving the supremum $\sup_{a,b\in(S^1\cup\{0\})}|a,x,b,y|$ cannot be zero and, by Theorem \ref{thm_pr_ineqs}(4), we will have
\begin{align*}
&\delta_{\B^2\backslash\{0\}}(x,y)
=\max\{\sup_{a,b\in S^1}\log(1+|a,x,b,y|),\,
\sup_{a\in S^1}\log(1+|a,x,0,y|)\}\\
&=\max\left\{\delta_{\B^2}(x,y),\,
\sup_{a\in S^1}\log\left(1+\frac{|x-y|}{|x-a||y|}\right)\right\}
=\max\left\{\rho_{\B^2}(x,y),\,
\log\left(1+\frac{|x-y|}{(1-|x|)|y|}\right)\right\}.
\end{align*}
\end{proof}

\begin{lemma}
For all distinct points $x,y\in\B^2\backslash\{0\}$ such that $|y|\leq|x|$,
\begin{align*}
\delta_{\B^2\backslash\{0\}}(x_0,y_0)
\leq\delta_{\B^2\backslash\{0\}}(x,y)
\leq\delta_{\B^2\backslash\{0\}}(x_1,y_1),
\end{align*}
when these points are as in Definition \ref{def_emr}, and $y_1\neq0$ and $|x_1|<1$ so that $\delta_{\B^2\backslash\{0\}}(x_1,y_1)$ is well-defined.
\end{lemma}
\begin{proof}
Let $x=qe^{\mu i}+k$ and $y=qe^{(\pi+\mu)i}+k$ for $k=(x+y)\slash2$, $q=|x-y|\slash2$ and $\mu\in[0,2\pi)$. Suppose without loss of generality that $k\in(r,1)$ and $0\leq\mu\leq\pi\slash2$. Clearly, $(x,y)=(x_0,y_0)$ if $\mu=\pi\slash2$ and $(x,y)=(x_1,y_1)$ if $\mu=0$. Since $\rho_{\B^2}(qe^{\mu i}+k,qe^{(\pi+\mu)i}+k)$ is decreasing with respect to $\mu$ as in the proof of Lemma \ref{lem_emrForDelta} and it can be shown by differentiation that so is the quotient
\begin{align*}
\frac{|x-y|}{(1-|x|)|y|}
=\frac{2q}{(1+\sqrt{k^2+q^2+2qk\cos(\mu)})\sqrt{k^2+q^2-2qk\cos(\mu)}},
\end{align*}
the result follows from Theorem \ref{thm_deltaPuncDisk}. 
\end{proof}

\begin{lemma}
For all $x,y\in\B^2\backslash\{0\}$,
\begin{align*}
\frac{1}{2}s_{\B^2\backslash\{0\}}(x,y)
\leq j^*_{\B^2\backslash\{0\}}(x,y)
\leq {\rm th}\frac{\delta_{\B^2\backslash\{0\}}(x,y)}{2}
\leq 2j^*_{\B^2\backslash\{0\}}(x,y)
\leq 2s_{\B^2\backslash\{0\}}(x,y)
\end{align*}
and the constants here are the best ones possible.
\end{lemma}
\begin{proof}
The inequality follows from \cite[Lemma 2.1, p. 1124 \& Lemma 2.2, p. 1125]{hvz} and Lemma \ref{thm_pr_ineqs}(2), and its sharpness from the fact that, for $h\to0^+$,
\begin{align*}
&\frac{{\rm th}(\delta_{R(r,1)}(h,-h)\slash2)}{s_{R(r,1)}(h,-h)}
=\frac{2}{2-h}
\to\frac{1}{2},
\quad
\frac{{\rm th}(\delta_{R(r,1)}(h,-h)\slash2)}{j^*_{R(r,1)}(h,-h)}
=\frac{1}{2-h}
\to1,\\
&\frac{{\rm th}(\delta_{R(r,1)}(1\slash2+h,1\slash2-h)\slash2)}{j^*_{R(r,1)}(1\slash2+h,1\slash2-h)}
=\frac{{\rm th}(\delta_{R(r,1)}(1\slash2+h,1\slash2-h)\slash2)}{s_{R(r,1)}(1\slash2+h,1\slash2-h)}
=\frac{2}{1+4h^2}
\to2.
\end{align*}
\end{proof}

\section{M\"obius metric in ring domains and capacity}

In this final section, we briefly study condenser capacity by using the M\"obius metric. Below, we introduce a new quantity that can be used to create lower bounds for the ring capacity of an arbitrary ring, see Lemma \ref{lemma_ringcap} and the preceding definitions. Note that it is useful to form bounds for the capacity of a condenser or a ring with the M\"obius metric rather than some other intrinsic metric, because the M\"obius invariance of this metric reflects the conformal invariance of the capacity at least partially.

\begin{definition}
For all disjoint non-empty sets $E,F\subset\overline{\R}^n$ with $\overline{E}\cap\overline{F}=\varnothing$, let $\delta(E,F)$ be the quantity $\delta_{\overline{\R}^n\backslash F}(E)$.
\end{definition}

It follows from Corollary \ref{cor_deltaEFsymmetric} below that the quantity $\delta(E,F)$ is truly symmetric but, in order to prove this theorem, we will need to consider another result first.

\begin{theorem}\label{thm_borderdeltadiam}
$(1)$ For a line segment $[u,v]$ in a domain $G\subset\R$ with ${\rm card}(\R\backslash G)\geq2$, $\delta_G([u,v])=\delta_G(u,v)$.\newline
$(2)$ For a compact set $E$ in a domain $G\subset\R^2$ with ${\rm card}(\R^2\backslash G)\geq2$, $\delta_G(E)=\delta_G(\partial E)$.\newline
$(3)$ For a compact set $E$ in a domain $G\subset\overline{\R}^n$ with ${\rm card}(\overline{\R}^n\backslash G)\geq2$, $\delta_G(E)=\delta_G(\partial E)$.
\end{theorem}
\begin{proof}
(1) By the M\"obius invariance of $\delta_G$, we may assume that $G$ is an interval $(a,b)\subset\R$ and, since the maximum value of the cross ratio $|a,x,b,y|$ for $a<u\leq x<y\leq v<b$ is trivially obtained with $x=u$ and $y=v$, the result follows.  

(2) Let $a,b\in\partial G$ and choose some points $x,y\in E$. Suppose that $y\notin\partial E$. Rotate $y$ around the point $x$ into a new point $y'\in E$ so that $|y'-b|$ is at minimum. Clearly, $|a,x,b,y'|>|a,x,b,y|$ because $|y'-b|$ decreases and all the other distances are preserved in this rotation. If $y'\in\partial E$, fix $y^*=y'$. If $y'\notin\partial E$, then $\measuredangle BXY'=0$. Fix now $y^*\in[y',b]\cap\partial E$ instead as in Figure \ref{fig5}. Note that the points are collinear in this case, and regardless whether they are on the line in the order $x,y',y^*,b$ or in the order $x,b,y^*,y'$, we will have $|a,x,b,y^*|>|a,x,b,y'|$. Thus, for any $y\notin\partial E$, we can always find a point $y^*\in\partial E$ such that $|a,x,b,y^*|>|a,x,b,y|$. Similarly, if $x\notin\partial E$, the point $x$ can be replaced with a suitable boundary point $x^*$. Thus, for all $a,b\in\partial G$, the maximum value of $|a,x,b,y|$ is obtained with points $x,y\in\partial E$ and the result follows. 

(3) By the M\"obius invariance of $\delta_G$, it is enough to prove the result in the case $G\subset\R^n$. If $n=1$, the result follows from the first part of this theorem. Suppose that $n\geq2$, fix $x,y\in E$ and let $a,b\in\partial G$ be the points giving the supremum $\sup_{a,b\in\partial G}|a,x,b,y|$. If $y\notin\partial E$, consider the plane containing $y,x,b$ and replace $y$ by a new point $y^*$ on this plane, just like above. If $x\notin\partial E$, choose similarly a new point on the plane with $x,a,y^*$. Since these changes cannot decrease the cross ratio, the result follows.
\end{proof}

\newcommand{\boundellipse}[3]
{(#1) ellipse (#2 and #3)
}

\begin{figure}[ht]
    \centering
    \begin{tikzpicture}[scale=3]
    \draw[thick]\boundellipse{0.15,1.8}{-0.5}{1.5};
    \draw[thick] (0.3,0.7) -- (0.5,3.7);
    \draw[thick] (0.3,0.7) circle (0.03cm);
    \draw[thick] (0.5,3.7) circle (0.03cm);
    \draw[thick] (0.39,2.05) circle (0.03cm);
    \draw[thick] (0.39,2.05) arc (86.188:110:1.353);
    \draw[thick] (0.3,0.7) -- (-0.163,1.971);
    \draw[thick] (-0.163,1.971) circle (0.03cm);
    \draw[thick] (0.45,3) circle (0.03cm);
    \draw[thick] (0.35,-0.1) circle (0.03cm);
    \draw[thick] (0.5,3.7) -- (0.5,3.7);
    \draw[thick] (0.4,3.9) arc (30:-33:3.923);
    \node[scale=1.3] at (0.2,3.65) {$G$};
    \node[scale=1.3] at (0,2.9) {$E$};
    \node[scale=1.3] at (0.4,0.8) {$x$};
    \node[scale=1.3] at (-0.1,2.1) {$y$};
    \node[scale=1.3] at (0.5,2.15) {$y'$};
    \node[scale=1.3] at (0.6,3.1) {$y^*$};
    \node[scale=1.3] at (0.49,-0.04) {$a$};
    \node[scale=1.3] at (0.57,3.84) {$b$};
    \end{tikzpicture}
    \caption{The point $y^*\in\partial E$ such that $|a,x,b,y^*|\geq|a,x,b,y|$ is found as in the proof of Theorem \ref{thm_borderdeltadiam} for $x,y\in E\subset G$ and $a,b\in\partial G$.}
    \label{fig5}
\end{figure}
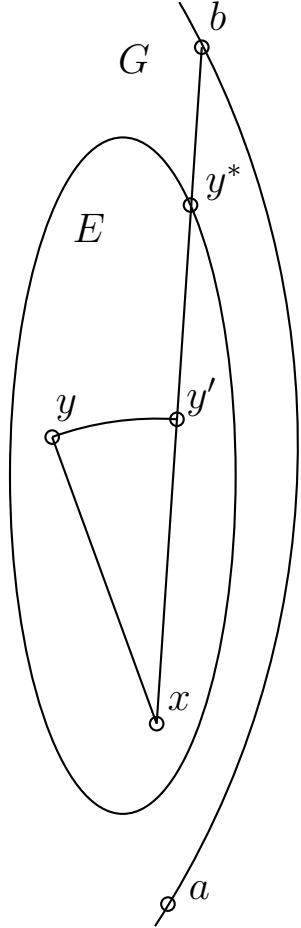

\begin{corollary}\label{cor_deltaEFsymmetric}
For all non-empty sets $E,F\subset\overline{\R}^n$ such that $\overline{E}\cap\overline{F}=\varnothing$, $\delta_{\overline{\R}^n\backslash F}(E)=\delta_{\overline{\R}^n\backslash E}(F)$.
\end{corollary}
\begin{proof}
It follows from Theorem \ref{thm_borderdeltadiam}(3), the symmetry of cross-ratio and the fact that $\partial F=\partial(\overline{\R}^n\backslash F)$ for all sets $F\subset\overline{\R}^n$ that
\begin{align*}
&\delta_{\overline{\R}^n\backslash F}(E)
=\delta_{\overline{\R}^n\backslash F}(\partial E)
=\sup_{x,y\in\partial E}\delta_{\overline{\R}^n\backslash F}(x,y)
=\sup_{x,y\in\partial E,\,a,b\in\partial F}\log(1+|a,x,y,b|)\\
&=\sup_{x,y\in\partial E,\,a,b\in\partial F}\log(1+|x,a,b,y|)
=\sup_{a,b\in\partial F}\delta_{\overline{\R}^n\backslash E}(a,b)
=\delta_{\overline{\R}^n\backslash E}(\partial F)
=\delta_{\overline{\R}^n\backslash F}(E).
\end{align*}
\end{proof}

Now that we have proved the symmetry of $\delta(E,F)$, let us study it further by finding its value in a few special cases.

\begin{theorem}
$(1)$ For the subsets $E=[-1,0]$ and $F=[s,\infty)$, $s>0$, of $\overline{\R}^n$, $\delta(E,F)=\log(1+1\slash s)$.\newline
$(2)$ If $E=[0,r]$ and $F=S^{n-1}$, $0<r<1$, instead, then $\delta(E,F)=2{\rm arth}(r)$.
\end{theorem}
\begin{proof}
(1) Let $x,y\in[-1,0]$ and $s\leq a<b$. Now, $y<a<b$ and $|a-b|<|y-b|$ so we need to choose $b=\infty$ to maximize $|a,x,b,y|$. The cross ratio $|a,x,\infty,y|$ is at greatest within limitations $x,y\in[-1,0]$ and $a\geq s>0$, when $a=s$, $x=0$ and $y=-1$. Consequently,
\begin{align*}
\delta(E,F)
=\delta_{\overline{\R}^n\backslash F}(E)
=\log(1+|s,0,\infty,-1|)
=\log(1+1\slash s).
\end{align*}
(2) By Theorem \ref{thm_pr_ineqs}(4),
\begin{align*}
\delta(E,F)
=\delta_{\overline{\R}^n\backslash F}(E)
=\delta_{\B^n}(E)
=\rho_{\B^n}(E)
=\rho_{\B^n}(0,r)
=2{\rm arth}(r).
\end{align*}
\end{proof}

We will yet introduce a few definitions that will be needed for Lemma \ref{lemma_ringcap}. The pair $(G,E)$ consisting of a domain $G\subset\R^n$ and a non-empty compact set $E\subset G$ is called a \emph{condenser}. Furthermore, if $G$ and $E$ are convex, then the set difference $G\backslash E$ is an example of a ring. More generally, a \emph{ring} is any domain $D\subset\overline{\R}^n$ whose complement $\overline{\R}^n\backslash D$ consist of exactly two components $C_0$ and $C_1$, and it can be denoted by $\mathcal{R}(C_0,C_1)$. The annular ring $R(r,1)$ with $0<r<1$ is a ring $\mathcal{R}(\overline{B}^2(0,r),\overline{\R}^n\backslash\B^2)$.

Define now the \emph{conformal capacity} of a condenser $(G,E)$ as
\begin{align}\label{condcapa}
\capa(G,E)=\inf_u\int_{G}|\nabla u|^n dm,
\end{align}
where infimum is taken over all functions $u\in C^\infty_0(G)$, $u: G\to[0,\infty)$ with $u(x)\geq1$ for all $x \in E$ and $dm$ stands for the $n$-dimensional Lebesgue measure. The definition \eqref{condcapa} can be also written as 
\begin{align*}
\capa(G,E)=\M(\Delta(E,\partial G;G)),
\end{align*}
where $\Delta(E,\partial G;G)$ is the family consisting of all such curves joining the set $E$ to the boundary $\partial G$ that are fully inside the domain $G$, and $\M(\Gamma)$ means the \emph{conformal modulus} of the curve family $\Gamma$, see \cite[6.1, p. 16]{v1} and \cite[pp. 103-106]{hkv}. Similarly, the \emph{capacity of a ring} $\mathcal{R}(C_0,C_1)$ is $\M(\Delta(C_0,C_1))$, where $\Delta(C_0,C_1)=\Delta(C_0,C_1;\overline{\R}^n)$.

Define yet the constant $c_n$, $n\geq2$, as in \cite[7.1.3, p. 114]{hkv}: 
\begin{align}\label{c_n}
c_n=\omega_{n-2}\left(2\int^{\pi\slash2}_0(\sin t)^{(2-n)\slash(n-1)} dt\right)^{1-n}
\geq\omega_{n-2}(\pi(n-1))^{1-n}
\quad\text{and}\quad
c_2=\frac{2}{\pi},   
\end{align}
where $\omega_{n-2}$ is the $(n-2)$-dimensional surface area of the sphere $S^{n-2}$.

\begin{lemma}\label{lemma_ringcap}
If $\mathcal{R}=\mathcal{R}(E,F)\subset\overline{\R}^n$ is a ring and $c_n$ as in \eqref{c_n}, then
\begin{align*}
\capa(\mathcal{R})\geq\frac{1}{2}c_n\delta(E,F).    
\end{align*}
\end{lemma}
\begin{proof}
By symmetry, let us suppose that $\infty\notin E$. By \cite[Lemma 9.29, p. 166]{hkv},
\begin{align*}
\capa(\mathcal{R})&\geq c_n\min\{j_{\R^n\backslash F}(E),j_{\R^n\backslash E}(F)\},\quad\text{if}\quad\infty\notin F,\quad\text{and}\\
\capa(\mathcal{R})&\geq c_nj_{\R^n\backslash F}(E),\quad\text{if}\quad\infty\in F.
\end{align*}
Regardless of whether $\infty\in F$ or not, we will have by Theorem \ref{thm_pr_ineqs}(1) and Corollary \ref{cor_deltaEFsymmetric},
\begin{align*}
\capa(\mathcal{R})
\geq\frac{1}{2}c_n\min\{\delta_{\R^n\backslash F}(E),\delta_{\R^n\backslash E}(F)\}
\geq\frac{1}{2}c_n\min\{\delta_{\overline{\R}^n\backslash F}(E),\delta_{\overline{\R}^n\backslash E}(F)\}
=\frac{1}{2}c_n\delta(E,F).
\end{align*}
\end{proof}

\begin{example}
Let $F$ be a compact subset of $\B^n$ and consider the ring $\mathcal{R}=\mathcal{R}(F,S^{n-1})$. There are two different lower bounds for the capacity of this ring: By Lemma \ref{lemma_ringcap} and Theorem \ref{thm_pr_ineqs}(4),
\begin{align}\label{lb0}
\capa(\mathcal{R})\geq\frac{1}{2}c_n\delta(F,S^{n-1})=\frac{1}{2}c_n\delta_{\overline{\R}^n\backslash S^{n-1}}(F)
=\frac{1}{2}c_n\delta_{\B^n}(F)
=\frac{1}{2}c_n\rho_{\B^n}(F)
\end{align}
and, by \cite[Lemma 7.13, p. 109; Lemma 9.20, p. 163]{hkv},
\begin{align}\label{lb1}
\capa(\mathcal{R})=\M(\Delta(F,S^{n-1}))
\geq\frac{1}{2}\M(\Delta(F,S^{n-1};\B^n))
\geq\frac{1}{2}\gamma_n\left(\frac{1}{\tnh(\rho_{\B^n}(F)\slash2)}\right),
\end{align}
where $\gamma_n$ is the Gr\"otzsch ring capacity function, see \cite[(7.17), p. 121]{hkv} for definition. If $n=2$, the lower bound \eqref{lb1} is better at least in some cases, but finding the exact value of the lower bound \eqref{lb0} is considerably easier since $c_2=2\slash\pi$.
\end{example}

\def\cprime{$'$} \def\cprime{$'$} \def\cprime{$'$}
\providecommand{\bysame}{\leavevmode\hbox to3em{\hrulefill}\thinspace}
\providecommand{\MR}{\relax\ifhmode\unskip\space\fi MR }
\providecommand{\MRhref}[2]{%
  \href{http://www.ams.org/mathscinet-getitem?mr=#1}{#2}
}
\providecommand{\href}[2]{#2}

\end{document}